\documentclass[12pt]{article}
\usepackage{amsmath,amsfonts,amssymb,amsthm,bbm}
\usepackage[latin1]{inputenc}
\usepackage[active]{srcltx}

\newtheorem{superclass}{superclass}
\newtheorem{definition}[superclass]{Definition}
\newtheorem{lemma}[superclass]{Lemma}
\newtheorem{theorem}[superclass]{Theorem}

\newtheorem{example}[superclass]{Example}

\newcommand{\R}{\mathbbm{R}}

\newcommand{\eps}{\varepsilon}
\newcommand{\mc}{\mathcal}

\DeclareMathOperator{\dist}{dist}

\DeclareMathOperator{\ext}{ext}
\DeclareMathOperator{\co}{conv}

\begin{document}

\title{Explicit Formula for Preimages of Relaxed One-Sided Lipschitz Mappings
with Negative Lipschitz Constants}
\author{Andrew S.~Eberhard, Boris S.~Mordukhovich and Janosch Rieger}
\date{\today}

\maketitle

\begin{abstract}
This paper addresses Lipschitzian stability issues that play an important role in both theoretical and numerical aspects of variational analysis, optimization, and their applications. We particularly concentrate on the so-called relaxed one-sided Lipschitz property of set-valued mappings with negative Lipschitz constants. This property has been much less investigated than more conventional Lipschitzian behavior while being well recognized in a variety of applications.
Recent work has revealed that set-valued mappings satisfying the relaxed one-sided Lipschitz condition with negative Lipschitz constant possess a localization property that is stronger than uniform metric regularity. The present paper complements this fact by providing a characterization not only of one specific single point of a preimage, but of entire preimages of such mappings. Developing a geometric approach, we derive an explicit formula to calculate preimages of relaxed one-sided Lipschitz mappings between finite-dimensional spaces and obtain a further specification of this formula via extreme points of image sets.
\end{abstract}

\noindent\textbf{Keywords:} Well-posedness in variational analysis, 
multivalued mapping, relaxed one-sided Lipschitz property, preimages, 
explicit formula\\
\noindent\textbf{MSC codes:} 49J53, 47H04

\section{Introduction}

The importance of various Lipschitzian stability and related properties of set-valued/multivalued mappings (multifunctions) has been highly recognized in numerous aspects of variational analysis, optimization, optimal control, and various of applications; see, e.g., the books \cite{rw,v,bz,Mor1,dr,m18} and the references therein. In this paper we consider the {\em relaxed one-sided Lipschitz property} which was introduced by Tzanko Donchev \cite{Donchev:96} and was used by him and other researchers as a stability criterion for ordinary differential inclusions. The relaxed one-sided Lipschitz constant of a multivalued vector field measures
the {\em rate of contraction or expansion} of the induced multivalued flow. Namely, a negative constant corresponds to contraction of mappings, while a positive constant corresponds to their expansion. Some modifications and further applications of the relaxed one-sided Lipschitz property of multifunctions can be found in, e.g., \cite{lv,dfm,rieger14} among other publications.

In this paper we entirely focus on the relaxed one-sided Lipschitz property of mappings with {\em negative} Lipschitz constants. Surjectivity results for such mappings are established in \cite{Donchev:02} and \cite{Donchev:04}. It is shown in \cite{Beyn:Rieger:10} that such mappings are (uniformly) {\em metrically regular}; see, e.g., \cite{rw,Mor1} for more discussions on this and related properties. The latter result has been recently refined in \cite{Beyn:Rieger:14} (see Theorem~\ref{solvability:theorem} below) and then further extended to infinite-dimensional Hilbert spaces and to the case of Gelfand triples in \cite{Rieger:Weth:16}.

To the best of our knowledge, descriptions of preimages/inverse images of sets (and in particular of singletons) under mappings satisfying the relaxed one-sided Lipschitz condition is a completely open area, while such objects play an important role in many aspects of variational analysis and optimization being significantly investigated for mappings with conventional Lipschitzian properties. The present paper partly albeit significantly fills this gap by deriving an {\em explicit characterization} of preimages of singletons under upper semicontinuous relaxed one-sided Lipschitz mappings with negative Lipschitz constant with providing further specifications via extreme points.

The paper is organized as follows. Section~\ref{prel} contains some preliminaries widely used below. In Section~\ref{basic:formula} we develop a new preimage formula for relaxed one-sided Lipschitz multifunctions, while Section~\ref{simplification} presents its major specification by using extreme points of images. The concluding Section~\ref{conc} discusses some challenging unsolved problems in this area.

\section{Preliminaries}\label{prel}

We equip $\R^d$ with the Euclidean norm $\|\cdot\|:\R^d\to\R_+$ and the canonical inner product $\langle\cdot,\cdot\rangle:\R^d\times\R^d\to\R$.
For any $z\in\R^d$ and an nonempty set $X\subset\R^d$ define the {\em distance function}
\[\dist(z,X)=\inf_{x\in X}\|x-z\|,\]
and for any $X\subset\R^d$ and $r\in\R_+$ we set
\[B_r(X)=B(X,r)=\big\{x\in\R^d:\,\dist(x,X)\le r\big\}.\]
The collection of all the nonempty {\em convex and compact} subsets of $\R^d$ is labeled as $\mc{CC}(\R^d)$. For any $\bar{y}\in\R^d$ and any mapping $F:\R^d\to\mc{CC}(\R^d)$ consider the {\em preimage set} given by
\[F^{-1}(\bar{y})=\big\{x\in\R^d:\,\bar{y}\in F(x)\big\}.\]
Consider further the {\em graph} of $F$ given by
$$
{\rm gph}\,F=\big\{(x,y)\in\R^d\times\R^d:\,y\in F(x)\big\}
$$

The main property of our study in this paper is defined as follows.

\begin{definition}[relaxed one-sided Lipschitz property]
Consider a set-valued mapping $F:\R^d\to\mc{CC}(\R^d)$ and fix a constant $\ell\in\R$. We say that $F$ is {\sc $\ell$-relaxed one-sided Lipschitz} if for any $x,x'\in\R^d$ and any $y\in F(x)$ there exists $y'\in F(x')$ such that
\[\langle y'-y,x'-x\rangle\le\ell\|x'-x\|^2.\]
\end{definition}

Recall that a set-valued mapping $F$ is {\em upper semicontinuous} if for every $x\in\R^d$ and $\eps>0$ there exists $\delta>0$ such that
$$
F(x')\subset B_\eps\big(F(x)\big)\;\mbox{ whenever }\;x'\in B_\delta(x).
$$

The following result is taken from \cite[Theorem~3.1]{Beyn:Rieger:14}.

\begin{theorem}{\rm(solvability property of relaxed one-sided Lipschitz mappings)}\label{solvability:theorem}
Let $F:\R^d\to\mc{CC}(\R^d)$ be an upper semicontinuous and $\ell$-relaxed one-sided Lipschitz mapping with Lipschitz constant $\ell<0$, let
$\bar{y}\in\R^d$, and let $x,y\in\R^d$ with $y\in F(x)$. 
Then there exists $\bar x\in\R^d$ with $(\bar x,\bar y)\in{\rm gph}\,F$ 
and such that
$$
\bar x\in B\Big(x+\tfrac{1}{2\ell}(\bar{y}-y),\tfrac{1}{2|\ell|}\|\bar y-y\|\Big).
$$
\end{theorem}

Note that the above statement is stronger than the aforementioned metric regularity, since it does not specify a ball centered at $x$ in which a solution $\bar{x}$ of $\bar{y}\in F(\bar{x})$ can be found, but a smaller ball with $x$ on its {\em boundary}.

\section{Explicit formula for preimages} \label{basic:formula}

Throughout this paper we fix the number $\ell<0$ as a given constant/modulus of the relaxed one-sided Lipschitz mapping $F:\R^d\to\mc{CC}(\R^d)$ under consideration, and also fix a pair $(\bar x,\bar y)\in{\rm gph}\,F$. Our goal is to derive a precise formula of representing the preimages of $F$ via the set-valued mapping defined by
\[x\mapsto G_F(x,\bar{y})=\bigcup_{y\in F(x)}B\Big(x+\frac{1}{2\ell}(\bar{y}-y),\frac{1}{2|\ell|}\|\bar{y}-y\|\Big).\]

We split the derivation of this formula into the following two lemmas. The first one requires only the relaxed one-sided Lipschitz property, while not upper semicontinuity of the mapping $F$.

\begin{lemma}{\rm(upper estimate of preimages)}\label{hue} Fix $\bar{y}\in\R^d$ and consider an $\ell$-relaxed one-sided Lipschitz mapping $F:\R^d\to\mc{CC}(\R^d)$ with $\ell<0$. Then we have the inclusion/upper estimate
\[F^{-1}(\bar{y})\subset G_F(x,\bar{y})\;\mbox{ for all }\;x\in\R^d.\]
\end{lemma}

\begin{proof}
Fix any $x\in\R^d$ and pick $\bar{x}\in F^{-1}(\bar{y})$. The relaxed one-sided Lipschitz property of $F$ allows us to find $y\in F(x)$ such that
\[\langle\bar{y}-y,\bar{x}-x\rangle\le\ell\|\bar{x}-x\|^2.\]
This implies the relationships
\begin{align*}
&\Big\|\bar{x}-(x+\frac{1}{2\ell}(\bar{y}-y))\Big\|^2=\Big\|(\bar{x}-x)-\frac{1}{2\ell}(\bar{y}-y)\Big\|^2\\
&=\|\bar{x}-x\|^2-\Big\langle\frac{1}{\ell}(\bar{y}-y),\bar{x}-x\Big\rangle+\Big\|\frac{1}{2\ell}(\bar{y}-y)\Big\|^2\le\Big\|\frac{1}{2\ell}(\bar{y}-y)\Big\|^2,
\end{align*}
which readily give us the inclusion
\[\bar{x}\in B\Big(x+\frac{1}{2\ell}(\bar{y}-y),\frac{1}{2|\ell|}\|\bar{y}-y\|\Big)\]
and thus verify the statement of the lemma.
\end{proof}

The reverse estimate requires only upper semicontinuity of $F$, while not its relaxed one-sided Lipschitz property.

\begin{lemma}{\rm(lower estimate of preimages)}\label{hott}
Let $F:\R^d\to\mc{CC}(\R^d)$ be upper semicontinuous, let $\ell<0$, and let $\bar{y}\in\R^d$.
If $z\notin F^{-1}(\bar{y})$, then there exists $x\in\R^d$ for which we get
\begin{equation}\label{lcni}
z\notin G_F(x,\bar{y}).
\end{equation}
\end{lemma}

\begin{proof}
Having $\bar{y}\notin F(z)$ and $F(z)\in\mc{CC}(\R^d)$, we find by the classical separation theorem such numbers $\eps>0$ and $v\in\R^d$ that
\begin{equation*}
\langle v,\bar{y}\rangle\le\langle v,y\rangle-2\eps\;\mbox{ for all }\;y\in F(z).
\end{equation*}
The upper semicontinuity of $F$ ensures the existence of $\delta>0$ with
\begin{equation}\label{separate}
\langle v,\bar{y}\rangle\le\langle v,y\rangle-\eps\;\mbox{ for all }\;y\in F(z+\delta v).
\end{equation}
Set $x:=z+\delta v$ and fix an arbitrary vector $y\in F(x)$. Using \eqref{separate} tells us that for every
$\xi\in\R^d$ with $\|\xi\|\le\frac{1}{2|\ell|}\|\bar{y}-y\|$,
we have
\begin{align*}
&\Big\langle\bar{y}-y,x+\frac{1}{2\ell}(\bar{y}-y)+\xi\Big\rangle
=\Big\langle\bar{y}-y,z+\delta v+\frac{1}{2\ell}(\bar{y}-y)+\xi\Big\rangle\\
&=\langle\bar{y}-y,z\rangle+\delta\Big\langle\bar{y}-y,v\Big\rangle
+\frac{1}{2\ell}\|\bar{y}-y\|^2+\langle\bar{y}-y,\xi\rangle
\le\langle\bar{y}-y,z\rangle-\delta\eps,
\end{align*}
and thus $z\ne x+\frac{1}{2\ell}(\bar{y}-y)+\xi$. This verifies that
\[z\notin B\Big(x+\frac{1}{2\ell}(\bar{y}-y),\frac{1}{2|\ell|}\|\bar{y}-y\|\Big).\]
Since $y\in F(x)$ was chosen arbitrarily, it gives us the conclusion in \eqref{lcni}.
\end{proof}

Combining Lemmas~\ref{hue} and \ref{hott}, we arrive at the aforementioned formula.

\begin{theorem}{\rm(precise formula for computing preimages)}\label{main:thm}
For any upper semicontinuous and $\ell$-relaxed one-sided Lipschitz 
mapping $F:\R^d\to\mc{CC}(\R^d)$ with $l<0$ and  for any $\bar{y}\in\R^d$ we have the preimage formula
\[F^{-1}(\bar{y})=\bigcap_{x\in\R^d}G_F(x,\bar{y}).\]
\end{theorem}

The following example shows that it is not necessary to take the intersection {\em over all of} $\R^d$ in the above formula, while it also demonstrates that there are {\em limits} to how far this result can potentially be improved.

\begin{example}\label{theoneandonly}
Consider the set-valued mapping $F:\R\to\mc{CC}(\R^d)$ given by
\[F(x)=\begin{cases}[1,2]-x,&x<0,\\ [-2,2],&x=0,\\ [-2,-1]-x,&x>0,\end{cases}\]
which is upper semicontinuous and satisfies the relaxed one-sided Lipschitz condition with the negative constant $\ell=-1$. We get $F^{-1}(0)=\{0\}$ and
\[G_F(x,0)=\begin{cases}[x,2],&x<0,\\ [-2,2],&x=0,\\ [-2,x],&x>0.\end{cases}\]
This readily yields the equality
\begin{equation}\label{unclear}
F^{-1}(0)=\cap_{x\in B_\eps(F^{-1}(0))}G_F(x,0)\;\mbox{ for all }\;\eps>0,
\end{equation}
At the same time we observe that
\begin{equation}\label{completely:unclear}
F^{-1}(0)\neq\cap_{x\in F^{-1}(0)}G_F(x,0).
\end{equation}
\end{example}

It is currently {\em unclear} whether formula \eqref{unclear} holds in general, and also whether an equality can be obtained in \eqref{completely:unclear} if assuming in addition that $F$ is continuous. The question on which vectors $y\in F(x)$ are really needed to compute $G_F(x,\bar{y})$ is addressed in the next section.

\section{Specification of the mapping $G_F$ in the preimage formula}\label{simplification}

Here we establish an effective specification of the mapping $G_F$ in the above preimage formula by using {\em extreme points} of the image sets. Given a compact and convex set $C$ in finite dimensions, the collection of all its extreme points is denoted by $\ext(C)$.

The proof of the main result of this section is based on the following four lemmas. The first one is a classical result of convex analysis, which can found, e.g., in \cite[Corollary~18.5.1]{Rockafellar:70}.

\begin{lemma}{\rm(existence of extreme points)}\label{ext:nonempty} If $C\in\mc{CC}(\R^d)$, then $\ext(C)\ne\emptyset$.
\end{lemma}

The next lemma provides a useful representation of ball covers via extreme points of the basic set. It leads us to a better understanding of images of the mapping $G_F$ in the preimage formula of Theorem~\ref{main:thm}.

\begin{lemma}{\rm(extreme points generate ball covers)}\label{play:with:balls}
Consider a set $C\in\mc{CC}(\R^d)$ and fix a vector $x\in\R^d$. If $x\in B_{\|z\|}(z)$ for some $z\in C$, then there exists an extreme point $z^*\in\ext(C)$ such that $x\in B_{\|z^*\|}(z^*)$.
\end{lemma}

\begin{proof}
For the fixed vector $x\in\R^d$, define the function $g:\R^d\to\R$ by
\[g(\xi):=\|x-\xi\|^2-\|\xi\|^2,\quad\xi\in\R^d.\]
We clearly have $x\in B_{\|\xi\|}(\xi)$ if and only if $\|x-\xi\|^2\le\|\xi\|^2$, which can be equivalently rewritten as
$g(\xi)\le 0$. Since the function
\[g(\xi)=\|x-\xi\|^2-\|\xi\|^2
=\|x\|^2-2\langle x,\xi\rangle\]
is linear in $\xi$, it attains its minimum over $C$ at an extreme point. It follows from the assumption made that there exists $z\in C$ with $x\in B_{\|z\|}(z)$. Thus
\[\min_{\xi\in\ext C}g(\xi)=\min_{\xi\in C}g(\xi)\le g(z)\le 0.\]
This ensures the existence of $z^*\in\ext C$ such that $g(z^*)\le 0$. The latter is clearly equivalent to
$x\in B_{\|z^*\|}(z^*)$, which completes the proof of the lemma.
\end{proof}

The following geometrical fact reveals a relationship between extreme points of the set $\ext(C)\cup\{0\}$ and its convexification denoted by ``conv".

\begin{lemma}{\rm(extreme points under convexification).}\label{cut:away}
For any set $C\in\mc{CC}(\R^d)$ we have the inclusion
$$
\ext\big(\co(C\cup\{0\})\big)\subset\ext(C)\cup\{0\}.
$$
\end{lemma}

\begin{proof}
Pick any $x\in\ext(\co(C\cup\{0\}))$ and find by definition a point $z\in C$ and a number $\mu\in[0,1]$ such that $x=\mu z$.
If $\mu=0$, then we get $x=0$, and thus the claimed statement holds. If $\mu\in(0,1)$, then the point $x=\mu z=\mu z+(1-\mu)0$ is not extreme for $\co(C\cup\{0\})$, which also verifies the statement of the lemma. It remains to examine the case where $\mu=1$. In this case we have $x\in C$.
If $x\notin\ext(C)$, then there exist $\lambda\in(0,1)$ and $z_1,z_2\in C$ with $x=\lambda z_1+(1-\lambda)z_2$ and $z_1\ne z_2$. Since $z_1,z_2\in\co(C\cup\{0\})$, it tells us that $x\notin\ext(\co(C\cup\{0\}))$ and thus completes the proof of this lemma.
\end{proof}

The last lemma of this section is based on the previous ones while providing a certain counterpart of Lemma~\ref{play:with:balls} under convexification. It actually shows that the extreme points of the set $C$ therein, which are facing towards zero, can be {\em omitted} in the statement of Lemma \ref{play:with:balls}.

\begin{lemma}{\rm(ball covers generated by outward facing extreme points)}\label{outward:facing} Take $C\in\mc{CC}(\R^d)$ and $x\in B_{\|z\|}(z)$ with some $z\in C$. Then there exists an extreme point $z^*\in\ext(C)\cap\ext(\co(C\cup\{0\}))$ such that we have $x\in B_{\|z^*\|}(z^*)$.
\end{lemma}

\begin{proof}
Consider the set $\widehat{C}:=\co(C\cup\{0\})\in\mc{CC}(\R^d)$. Since we know that $x\in B_{\|z\|}(z)$ for some $z\in C\subset\widehat{C}$,
applying Lemma~\ref{play:with:balls} to $x$, $z$ and the set $\hat{C}$ together with the usage of Lemma~\ref{cut:away} tells us that $x\in B_{\|\hat{z}\|}(\hat{z})$ for some
\[\hat{z}\in\ext(\widehat{C})=\ext(\co(C\cup\{0\}))\subset\ext(C)\cup\{0\}.\]
If $\hat{z}\ne 0$, then $\hat{z}\in\ext(C)$, and the claimed statement is fulfilled with $z^*:=\hat{z}$. If $\hat{z}=0$, then we have $x=0$.
Employing Lemma~\ref{ext:nonempty} yields $z^*\in\ext(C)$, which verifies therefore that $x\in B_{\|z^*\|}(z^*)$ while completing the proof.
\end{proof}

Applying now Lemma~\ref{outward:facing} to the problem in question leads us to the following specification of the images of the mapping $G_F(x,\bar{y})$ for each fixed $\bar y\in\R^d$ and hence of the preimages of $F$ in the calculation formula of Theorem~\ref{main:thm}. This result shows that for such a calculation we may use {\em only extreme points} of $F(x)$ {\em facing away} from $\bar{y}$.

\begin{theorem}{\rm(specified calculation formula via extreme points).}\label{spec} For any set-valued mapping $F:\R^d\to\mc{CC}(\R^d)$  satisfying the relaxed one-sided Lipschitz property with a negative Lipschitz constant $\ell<0$, any $x\in\R^d$, and any $\bar{y}\in\R^d$, we have the representation
\begin{align*}
G_F(x,\bar{y})=\bigcup_{y\in\ext(F(x))\cap\ext(\co(F(x)\cup\{\bar{y}\}))}
B\Big(x+\frac{1}{2\ell}(\bar{y}-y),\frac{1}{2|\ell|}\|\bar{y}-y\|\Big).
\end{align*}
\end{theorem}

\begin{proof}
The inclusion ``$\supset$" is trivial. Let us verify the converse inclusion therein. Applying Lemma~\ref{outward:facing} to the set $C:=\frac{1}{2\ell}(\bar{y}-F(x))$ yields
\begin{equation}\label{to:be:shifted}
\bigcup_{z\in C}B_{\|z\|}(z)\subset\bigcup_{z\in\ext(C)\cap\ext(\co(C\cup\{0\}))}
B_{\|z\|}(z).
\end{equation}
It can  be directly checked that
\begin{align*}
\bigcup_{z\in C}B_{\|z\|}(z)
=\bigcup_{z\in\frac{1}{2\ell}(\bar{y}-F(x))}B_{\|z\|}(z)
=\bigcup_{y\in F(x)}B\Big(\frac{1}{2\ell}(\bar{y}-y),\frac{1}{2|\ell|}\|\bar{y}-y\|\Big).
\end{align*}
In addition we have the equalities
\begin{align*}
&\ext(C)\cap\ext\big(\co(C\cup\{0\})\big)\\
&=\ext\Big(\frac{1}{2\ell}(\bar{y}-F(x))\Big)
\cap\ext\Big(\co\big(\frac{1}{2\ell}(\bar{y}-F(x))\cup\{0\}\big)\Big)\\
&=\frac{1}{2\ell}\Big(\ext\big(\bar{y}-F(x)\Big)
\cap\ext\Big(\co\big((\bar{y}-F(x))\cup\{0\})\big)\Big)\\
&=\frac{1}{2\ell}\Big(\bar{y}-\Big\{\ext\big(F(x)\big)
\cap\ext\big(\co(F(x)\cup\{\bar{y}\})\big)\Big\}\Big).
\end{align*}
Therefore, the following relationship is satisfied:
\begin{align*}
&\bigcup_{z\in\ext(C)\cap\ext(\co(C\cup\{0\}))}B_{\|z\|}(z)\\
&=\bigcup_{y\in\ext(F(x))\cap\ext(\co(F(x)\cup\{\bar{y}\}))}
B\Big(\frac{1}{2\ell}(\bar{y}-y),\frac{1}{2|\ell|}\|\bar{y}-y\|\Big).
\end{align*}
Combining all the above and shifting inclusion \eqref{to:be:shifted} by $x$, we arrive at the claimed statement of the theorem.
\end{proof}

\section{Conclusions}\label{conc}

The results of this paper provide useful formulas to calculate preimages of points under set-valued mappings satisfying the relaxed one-sided Lipschitz property with negative Lipschitz constants. The importance of such mappings has been recognized in variational analysis, optimization, optimal control, and stability theory for differential equations and inclusions, while this property is largely underinvestigated. In particular, regularity and the topological features of the preimages of relaxed one-sided Lipschitz mappings with negative Lipschitz constants are almost completely unexplored. The explicit geometric characterization obtained in this paper is only a first step in this direction, and much more work should be done in the future. Note furthermore that, in contrast to well-developed generalized differential calculus for various compositions involving preimages of sets under Lipschitz continuous multifunctions and the like (see, e.g., \cite{rw,Mor1,m18}), nothing at all is known in the case of mappings satisfying the relaxed one-sided Lipschitzian condition considered in our paper. This definitely is a challenging issue of variational analysis with many potential applications to optimization, control, and other areas.

\subsection*{Acknowledgements}
Research of the second author was partly supported by the USA National Science Foundation under grants DMS-1512846 and DMS-1808978, by the USA Air Force Office of Scientific Research under grant \#15RT0462, and by the Australian Research Council under Discovery Project DP-190100555.

\end{document}